\documentclass[a4paper]{amsart}

\usepackage{amsfonts}
\usepackage{amssymb}
\usepackage{amsmath}
\usepackage{hyperref}
\usepackage{mathrsfs}
\usepackage{centernot}
\usepackage{mathdots}
\usepackage{stmaryrd}
\usepackage[all]{xy}

\newtheorem{thm}{Theorem}[section]

\newtheorem{prp}[thm]{Proposition}
\newtheorem{cor}[thm]{Corollary}
\newtheorem{dfn}[thm]{Definition}

\newtheorem{baseexample}[thm]{Example} %never to be used !!! - use example environment
\newtheorem{baseremark}[thm]{Remark} %never to be used !!! - use remark environment

\newenvironment{example}
{\begin{baseexample}\rm}{\end{baseexample}}
\newenvironment{remark}
{\begin{baseremark}\rm}{\end{baseremark}}

\newcommand{\rem}[1]{}

\newcommand{\N}{\mathbb{N}}
\newcommand{\Q}{\mathbb{Q}}
\newcommand{\R}{\mathbb{R}}
\newcommand{\Z}{\mathbb{Z}}

\newcommand{\frakCapital}{
\newcommand{\frakA}{{\mathfrak{A}}}
\newcommand{\frakB}{{\mathfrak{B}}}
\newcommand{\frakC}{{\mathfrak{C}}}
\newcommand{\frakD}{{\mathfrak{D}}}
\newcommand{\frakE}{{\mathfrak{E}}}
\newcommand{\frakF}{{\mathfrak{F}}}
\newcommand{\frakG}{{\mathfrak{G}}}
\newcommand{\frakH}{{\mathfrak{H}}}
\newcommand{\frakI}{{\mathfrak{I}}}
\newcommand{\frakJ}{{\mathfrak{J}}}
\newcommand{\frakK}{{\mathfrak{K}}}
\newcommand{\frakL}{{\mathfrak{L}}}
\newcommand{\frakM}{{\mathfrak{M}}}
\newcommand{\frakN}{{\mathfrak{N}}}
\newcommand{\frakO}{{\mathfrak{O}}}
\newcommand{\frakP}{{\mathfrak{P}}}
\newcommand{\frakQ}{{\mathfrak{Q}}}
\newcommand{\frakR}{{\mathfrak{R}}}
\newcommand{\frakS}{{\mathfrak{S}}}
\newcommand{\frakT}{{\mathfrak{T}}}
\newcommand{\frakU}{{\mathfrak{U}}}
\newcommand{\frakV}{{\mathfrak{V}}}
\newcommand{\frakW}{{\mathfrak{W}}}
\newcommand{\frakX}{{\mathfrak{X}}}
\newcommand{\frakY}{{\mathfrak{Y}}}
\newcommand{\frakZ}{{\mathfrak{Z}}}
}

\newcommand{\calCapital}{
\newcommand{\calA}{{\mathcal{A}}}
\newcommand{\calB}{{\mathcal{B}}}
\newcommand{\calC}{{\mathcal{C}}}
\newcommand{\calD}{{\mathcal{D}}}
\newcommand{\calE}{{\mathcal{E}}}
\newcommand{\calF}{{\mathcal{F}}}
\newcommand{\calG}{{\mathcal{G}}}
\newcommand{\calH}{{\mathcal{H}}}
\newcommand{\calI}{{\mathcal{I}}}
\newcommand{\calJ}{{\mathcal{J}}}
\newcommand{\calK}{{\mathcal{K}}}
\newcommand{\calL}{{\mathcal{L}}}
\newcommand{\calM}{{\mathcal{M}}}
\newcommand{\calN}{{\mathcal{N}}}
\newcommand{\calO}{{\mathcal{O}}}
\newcommand{\calP}{{\mathcal{P}}}
\newcommand{\calQ}{{\mathcal{Q}}}
\newcommand{\calR}{{\mathcal{R}}}
\newcommand{\calS}{{\mathcal{S}}}
\newcommand{\calT}{{\mathcal{T}}}
\newcommand{\calU}{{\mathcal{U}}}
\newcommand{\calV}{{\mathcal{V}}}
\newcommand{\calW}{{\mathcal{W}}}
\newcommand{\calX}{{\mathcal{X}}}
\newcommand{\calY}{{\mathcal{Y}}}
\newcommand{\calZ}{{\mathcal{Z}}}
}

\newcommand{\bbCapital}{
\newcommand{\bbA}{{\mathbb{A}}}
\newcommand{\bbB}{{\mathbb{B}}}
\newcommand{\bbC}{{\mathbb{C}}}
\newcommand{\bbD}{{\mathbb{D}}}
\newcommand{\bbE}{{\mathbb{E}}}
\newcommand{\bbF}{{\mathbb{F}}}
\newcommand{\bbG}{{\mathbb{G}}}
\newcommand{\bbH}{{\mathbb{H}}}
\newcommand{\bbI}{{\mathbb{I}}}
\newcommand{\bbJ}{{\mathbb{J}}}
\newcommand{\bbK}{{\mathbb{K}}}
\newcommand{\bbL}{{\mathbb{L}}}
\newcommand{\bbM}{{\mathbb{M}}}
\newcommand{\bbN}{{\mathbb{N}}}
\newcommand{\bbO}{{\mathbb{O}}}
\newcommand{\bbP}{{\mathbb{P}}}
\newcommand{\bbQ}{{\mathbb{Q}}}
\newcommand{\bbR}{{\mathbb{R}}}
\newcommand{\bbS}{{\mathbb{S}}}
\newcommand{\bbT}{{\mathbb{T}}}
\newcommand{\bbU}{{\mathbb{U}}}
\newcommand{\bbV}{{\mathbb{V}}}
\newcommand{\bbW}{{\mathbb{W}}}
\newcommand{\bbX}{{\mathbb{X}}}
\newcommand{\bbY}{{\mathbb{Y}}}
\newcommand{\bbZ}{{\mathbb{Z}}}
}

\newcommand{\catCapital}{
\newcommand{\catA}{{\mathscr{A}}}
\newcommand{\catB}{{\mathscr{B}}}
\newcommand{\catC}{{\mathscr{C}}}
\newcommand{\catD}{{\mathscr{D}}}
\newcommand{\catE}{{\mathscr{E}}}
\newcommand{\catF}{{\mathscr{F}}}
\newcommand{\catG}{{\mathscr{G}}}
\newcommand{\catH}{{\mathscr{H}}}
\newcommand{\catI}{{\mathscr{I}}}
\newcommand{\catJ}{{\mathscr{J}}}
\newcommand{\catK}{{\mathscr{K}}}
\newcommand{\catL}{{\mathscr{L}}}
\newcommand{\catM}{{\mathscr{M}}}
\newcommand{\catN}{{\mathscr{N}}}
\newcommand{\catO}{{\mathscr{O}}}
\newcommand{\catP}{{\mathscr{P}}}
\newcommand{\catQ}{{\mathscr{Q}}}
\newcommand{\catR}{{\mathscr{R}}}
\newcommand{\catS}{{\mathscr{S}}}
\newcommand{\catT}{{\mathscr{T}}}
\newcommand{\catU}{{\mathscr{U}}}
\newcommand{\catV}{{\mathscr{V}}}
\newcommand{\catW}{{\mathscr{W}}}
\newcommand{\catX}{{\mathscr{X}}}
\newcommand{\catY}{{\mathscr{Y}}}
\newcommand{\catZ}{{\mathscr{Z}}}
}

\newcommand{\vphi}{\varphi}

\newcommand{\suchthat}{\,:\,}
\newcommand{\where}{\,|\,}

\newcommand{\quo}[1]{\overline{#1}}

\newcommand{\SMatII}[4]{\left[\begin{array}{cc} {#1} & {#2} \\ {#3} &
{#4} \end{array}\right]}
\newcommand{\smallSMatII}[4]{\left[\begin{smallmatrix} {#1} & {#2} \\ {#3} &
{#4} \end{smallmatrix}\right]}

\newcommand{\Circs}[1]{\left( #1 \right)}

\DeclareMathOperator{\Cent}{Cent}

 \DeclareMathOperator{\End}{End}
 
\DeclareMathOperator{\Hom}{Hom} %
\DeclareMathOperator{\id}{id} %
\DeclareMathOperator{\im}{im} %
\DeclareMathOperator{\Jac}{Jac} %

\newcommand{\op}{\mathrm{op}} %
\DeclareMathOperator{\ord}{ord} %
\newcommand{\nMat}[2]{\mathrm{M}_{#2}(#1)}

\newcommand{\invlim}{\underleftarrow{\lim}\,}

\newcommand{\units}[1]{{#1^\times}}

\newcommand{\rMod}[1]{{\mathrm{Mod}\textrm{-}{#1}}}

\newcommand{\lMod}[1]{{#1}\textrm{-}{\mathrm{Mod}}}

\newcommand{\rproj}[1]{{\mathrm{proj}}\textrm{-}{#1}}
\newcommand{\lproj}[1]{{#1}\textrm{-}{\mathrm{proj}}}

\usepackage{longtable}
\usepackage[foot]{amsaddr}

\calCapital %
\frakCapital %
\bbCapital %
\catCapital %

\DeclareMathOperator{\Rep}{Rep}

\newcommand{\rHMod}[1]{{\mathrm{HMod}\textrm{-}{#1}}}
\newcommand{\FP}[1]{{\mathrm{fp}\textrm{-}{#1}}}
\newcommand{\HFP}[1]{{\mathrm{hfp}\textrm{-}{#1}}}
\newcommand{\lHFP}[1]{{{#1}\textrm{-}\mathrm{hfp}}}

\title{Categorical Realizations of Quivers}

\author{Uriya A.\ First$^*$}
\date{\today}
\address{$^*$Einstein Institute of Mathematics, Hebrew University of Jerusalem}
\email{uriya.first@gmail.com}
\thanks{This research was supported by an ERC grant \#226135 and by
the Lady Davis Fellowship Trust.}

\begin{document}

\maketitle

\begin{abstract}
    We introduce and study categorical realizations of quivers.
    This construction generalizes \emph{comma categories}
    and includes representations of quivers on categories, twisted representations of quivers
    (in the sense of \cite{GothKing05}) and bilinear pairings as special cases.
    We prove a Krull-Schmidt Theorem in this general context, which results in a Krull-Schmidt
    Theorem for the special cases just mentioned. We also show that cancelation holds
    under milder assumptions. Using similar ideas we  prove
    a version of Fitting's Lemma for natural transformations between functors.
\rem{
    Let $\catC$ be an additive category and let $Q$ be a quiver.
    We show that under certain assumptions,
    every representation of $Q$ on $\catC$ has a \emph{Krull-Schmidt decomposition}.
    For example, this holds when $\catC$ is a category of finitely presented right $R$-modules over
    a semilocal ring $R$  which is complete in the $\Jac(R)$-adic topology
    and such that $\Jac(R)$ is finitely generated as a right ideal.
    We also obtain a version of Fitting's Lemma for natural transformations
    between functors $F:\catU\to \catC$, where $\catU$ is an arbitrary category.
}
\end{abstract}

\section{Overview}

    Throughout, all rings are assumed to have a unity and ring
    homomorphisms are required to preserve it.
    Subrings are assumed to have the same unity as the ring containing them.
    The Jacobson radical of a ring $R$ is denoted by $\Jac(R)$.
    %A \emph{semisimple} ring always means a semisimple artinian ring.

    \medskip

    In this paper, we introduce categorical realizations of quivers and representations
    of quivers on these realizations. This construction generalizes \emph{comma categories} (which are a categorical
    realizations of the quiver $\bullet\!\to\!\bullet\,$; see \cite[3K]{AdHerStr90AbstractAndConcreteCats}).
    Among its special cases are representations of quivers on additive categories (vector spaces
    in particular), twisted representations of quivers in the sense of \cite{GothKing05},  and  bilinear pairings.

    Using \emph{semi-centralizer subrings} (also called
    \emph{semi-invariant subrings}), introduced in \cite{Fi12C}, we establish a \emph{Krull-Schmidt
    Theorem} for representations of quivers on their realizations, thus yielding Krull-Schmidt Theorems
    for  the special cases mentioned above.
    We note that the categories in question
    are usually not abelian, so no ``finite length'' considerations (see Example~\ref{REP:EX:abelian-cats}) can be applied.
    Using similar ideas, we also prove a version of Fitting's Lemma for natural transformations
    between functors.
    Finally, different techniques are used to prove
    cancelation (from direct sums) of representations of quivers on their realizations under milder assumptions.
    Again, this yields cancelation theorems for the previous special cases.

    This paper can be viewed as a continuation of \cite{Fi12C} presenting further applications
    of semi-invariant subrings.

    \medskip

    Section~\ref{section:Krull-Schmidt} recalls the Krull-Schmidt Theorem for pseudo-abelian categories.
    Section~\ref{section:realizations} introduces categorical realizations of quivers
    and section~\ref{section:LT} defines linearly topologized categories.
    Section~\ref{section:pi-regular} recalls Fitting's Property (called \emph{quasi-$\pi_\infty$-regular} in \cite{Fi12C})
    and semi-centralizer subrings.
    The main result of the paper and its applications are presented in section~\ref{section:main}.
    A version of Fitting's Lemma for natural transformations is given
    in section \ref{section:Fitting}. Finally, section~\ref{section:cancelation} discusses cancelation from direct sums.

\section{The Krull-Schmidt Theorem}
\label{section:Krull-Schmidt}

    Let $\catC$ be an additive category. We say that the \emph{idempotents of $\catC$ split}
    or that $\catC$ is \emph{pseudo-abelian}
    if for all $C\in\catC$ and an idempotent $e\in\End_{\catC}(C)$, there exists
    an object $S\in\catC$ and morphisms $i:S\to C$, $p:C\to S$ such that $i\circ p=e$ and $p\circ i=\id_S$. In this
    case, $(S,i,p)$ is called a \emph{summand} of $C$. Every additive category is a full subcategory of
    a pseudo-abelian category obtained by artificially adding the nonexisting summands. This
    category is called the \emph{Karubi envelope} of $\catC$; see \cite[Th.\ 6.10]{Kar78} for details.

    An object of $C\in\catC$ is said to have a Krull-Schmidt decomposition if $C=C_1\oplus\dots\oplus C_t$
    with $C_1,\dots,C_t$ indecomposable, and for every other such decomposition $C=C'_1\oplus\dots\oplus C'_{t'}$,
    we have $t=t'$ and $C_i\cong C'_{i}$ after suitable reordering of the $C_i$-s.
    We say that a
    pseudo-abelian category $\catC$ is a \emph{Krull-Schmidt category} (abbrev.: KS category) if every
    object in $\catC$ has a Krull-Schmidt decomposition.

    \smallskip

    Recall that a ring $R$ is  \emph{semiperfect} if $R/\Jac(R)$ is semilocal and $\Jac(R)$ is idempotent
    lifting (e.g.\ if $\Jac(R)$ is nil). Equivalently, a ring $R$
    is semiperfect if there are pairwise orthogonal idempotents $e_1,\dots,e_t$ such that $\sum_ie_i=1$
    and $e_iRe_i$ is local for all $i$; see \cite[\S2.7]{Ro88}.
    For example, if $R$ is semilocal and complete in the
    $\Jac(R)$-adic topology, i.e.\ $R=\invlim \{R/\Jac(R)^n\}_{n\in\N}$, then $R$ is semiperfect.

    Let $\catC$ be a pseudo-abelian category.
    The \emph{Krull-Schmidt Theorem} (which is also due to Azumaya and Remak; see \cite[Th.~2.9.17~ff.]{Ro88}) asserts
    that an object $C\in\catC$ has a Krull-Schmidt Decomposition when $\End_{\catC}(C)$
    is semiperfect. In particular, if all objects in $\catC$ have a semiperfect endomorphism ring,
    then $\catC$ is a KS category.

    \begin{example}\label{REP:EX:abelian-cats}
        Abelian categories in which all objects have finite length are KS categories.
        This follows from the Hadara-Sai Lemma (e.g.\ see \cite[Prp.\ 2.9.29]{Ro88}\footnote{
            The proof is phrased for modules
            but works in arbitrary abelian categories.
        }) which implies that the endomorphism ring of every indecomposable
        object of finite length is semiprimary.
    \end{example}

\section{Categorical Realizations of Quivers}
\label{section:realizations}

    In this section, we introduce categorical realizations of quivers.
    Many known objects are special cases of this construction: representations of quivers on categories,
    twisted representations of quivers in the sense of \cite{GothKing05} and related papers,
    the \emph{comma category} (e.g.\ see \cite[\S3K]{AdHerStr90AbstractAndConcreteCats}), and even systems of bilinear pairings.

\rem{
    Twisted representations of quivers were treated in \cite{GothKing05} and related papers.
    In this paper, we shall work with a more general definition,
    which we  describe in this section. The definition of \cite{GothKing05} as well as ``non-twisted'' representations and several
    other examples will be recovered as special cases of our definition.
}

    \medskip

    Recall that a \emph{quiver} $Q$ is a  directed graph (possibly with loops and multiple edges).
    We let $Q_0$ (resp.\ $Q_1$) denote  the set of vertices
    (resp.\ directed edges) in $Q$. \emph{We always assume $Q_0$ is finite, but do not
    require $Q_1$ to be finite.}
    The source (resp.\ target) of an edge $a\in Q_1$
    is denoted by $s(a)$ (resp.\ $t(a)$).
    A path $p$ in $Q$ consists of a sequence $(v_1,a_1,v_2,a_2,\dots,a_{n-1},v_n)$ ($n\geq 1$)
    where $v_1,\dots,v_n\in Q_0$ and $a_i$ is an edge from $v_i$ to $v_{i+1}$.
    In this case, we  let $s(p)=v_1$ and $t(p)=v_n$.
    The concatenation of two paths is defined in the obvious way.

    Fix a quiver $Q$ and assume
    that for
    every vertex $v\in Q_0$,
    we are given an additive category $\catC_v$, and for
    every arrow $a\in Q_1$, we are given an additive category $\catC_a$ and two additive
    functors $S_a:\catC_{s(a)}\to \catC_a$, $T_a:\catC_{t(a)}\to\catC_a$. The pair
    \[\catR:=\Circs{\{\catC_x\}_{x\in Q_0\cup Q_1}, \{S_a,T_a\}_{a\in Q_1}}\]
    is called a \emph{categorical realization} of $Q$.

    A representation of $Q$ on $\catR$
    consists
    of a collection of objects $\{C_v\}_{v\in Q_0}$ and morphisms
    $\{f_{a}\}_{a\in Q_1}$ such that $C_v\in\catC_v$ and $f_a\in\Hom_{\catC_a}(S_aC_{s(a)},T_aC_{t(a)})$
    for all $v\in Q_0$ and $a\in Q_1$.
    Denote by $\Rep(Q,\catR)$ the category of  representations of $Q$ on $\catR$.
    A morphism from $(\{C_v\},\{f_a\})$ to $(\{C'_v\},\{f'_a\})$ is a family of
    morphisms $\{\psi_v\}_{v\in Q_0}$ such that $\psi_v\in\Hom_{\catC_v}(C_v,C'_v)$, and  for all $a\in Q_1$,
    the following diagram commutes:
    \[
    \xymatrix{
    S_aC_{s(a)} \ar[r]^{f_a} \ar[d]_{S_a\psi_{s(a)}} & T_aC_{t(a)} \ar[d]^{T_a\psi_{t(a)}} \\
    S_aC'_{s(a)} \ar[r]^{f'_a} & T_aC'_{t(a)}
    }
    \]

    \begin{example}\label{REP:EX:examples-of-quivers}
        When the categories $\{\catC_x\}_{x\in Q_0\cup Q_1}$
        are equal to a fixed additive category $\catC$ and $S_a=T_a=\id_{\catC}$ for all $a\in Q_1$, we get the notion
        of a representation of $Q$ on $\catC$. The realization $\catR$ is then called
        the \emph{trivial realization} of $Q$ on $\catC$, and we  write $\Rep(Q,\catC)$
        instead of $\Rep(Q,\catR)$. When $\catC$ is the category
        of finite dimensional vector spaces over a field $F$, the representations
        of $Q$ on $\catC$ are just representations of $Q$ over the field $F$ in the classical sense (see
        \cite{AssemIbSim06Reps} for an extensive discussion about this  case).

        If $Q_0$ consists of a single vertex and $Q_1$ consists of $n$ arrows, then the
        isomorphism classes of
        $\Rep(Q,\catC)$ correspond to tuples of endomorphisms $(f_1,\dots,f_n)\in\End_{\catC}(C)$,
        considered up to simultaneous conjugacy. In fact, for $n=1$, $\Rep(Q,\catC)$
        is equivalent to the category of \emph{endomorphisms in $\catC$}.

        If $Q_0$ consists of two vertices $v_1,v_2$ and $Q_1$ consists of $n$ edges  from
        $v_1$ to $v_2$, then the isomorphism classes of $\Rep(Q,\catC)$ correspond
        to tuples of homomorphisms $(f_1,\dots,f_n)\in\Hom_{\catC}(C,C')$ considered up to
        simultaneous composition on the left and on the right by an isomorphism in $\catC$.
        In fact, when $n=1$, $\Rep(Q,\catC)$ is equivalent the category of \emph{morphisms in $\catC$}.
    \end{example}

    \begin{example}
        Let $Q$ be the quiver $u\xrightarrow{a} v$ and let $\catR$ be a categorical
        realization of $Q$. Then $\Rep(Q,\catR)$ is just the \emph{comma category}
        $S_a\downarrow T_a$ (see \cite[\S3K]{AdHerStr90AbstractAndConcreteCats}). The construction of $\Rep(Q,\catR)$
        is therefore a  generalization of comma categories.
    \end{example}

    \begin{example}\label{REP:EX:twisted-quivers}
        Let $R$ be a commutative ring and let $M:=\{M_a\}_{a\in Q_1}$ be
        a collection of right $R$-modules.
        Take $\catC_v=\catC_a=\rMod{R}$ and define $S_a(X)=X\otimes_R M$, $T_a(X)=X$
        for all $v\in Q_0$, $a\in Q_1$, $X\in\rMod{R}$. Then $\Rep(Q,\catR)$ is the
        category of $M$-twisted representations of $Q$ on $\rMod{R}$ in the sense of \cite[\S2]{GothKing05}.
        This example can be generalized by replacing $R$ with a ringed space $(X,\calO_X)$ and
        $\rMod{R}$ with the category of sheaves of $\calO_X$-modules.
    \end{example}

    \begin{example}\label{REP:EX:bilinear-pairings}
        Let $R$ and $S$ be two  rings and let $M$ be  a fixed $(R,S)$-bimodule.
        An \emph{$(R,S)$-bilinear pairing} (taking values in $M$)
        is a biadditive map $\omega:A\times B\to M$ with $A\in\lMod{R}$, $B\in\rMod{S}$
        satisfying
        $\omega(ra,bs)=r\cdot\omega(a,b)\cdot s$ for all $a\in A$, $b\in B$, $r\in R$, $s\in S$.
        Denote by $\catB$ the class of $(R,S)$-bilinear pairings. We make $\catB$ into
        a category by defining a homomorphism  from $\omega:A\times B\to M$ to $\omega':A'\times B'\to M$ to
        be a pair $(f,g)\in\Hom_R(A,A')\times \Hom_S(B',B)$ such
        that
        \[
        \omega'(fa,b')=\omega(a,gb')\qquad\forall~ a\in A,~b'\in B'\ .
        \]
        %(homomorphisms of left $R$-modules are applied from the right).
        Composition is defined
        by the formula $(f',g')\circ (f,g)=(f'\circ f,g'\circ g)$. While
        the morphisms in $\catB$ are not the standard morphisms
        of bilinear pairings, the isomorphism classes in $\catB$
        are still isomorphism classes of bilinear pairing in the usual sense.

        We claim that the category $\catB$ is equivalent
        $\Rep(Q,\catR)$ for suitable $Q$ and $\catR$. Indeed, take $Q$ to be the quiver
        \mbox{$u\xrightarrow{~a~} v$}, and define $\catR$ by putting $\catC_u=\lMod{R}$, $\catC_v=(\rMod{S})^\op$, $\catC_a=\lMod{R}$,
        $S_a(X)=X$ and $T_a(X)=\Hom_S(X,M)$. An object in $\Rep(Q,\catR)$
        consists of a morphism of $R$-modules $A\to \Hom_S(B,M)$, whose datum is equivalent
        to an $(R,S)$-bilinear pairing $A\times B\to M$.

        It is also possible to treat systems of $(R,S)$-bilinear pairings $\{\omega_i:A\times B\to M_i\}_{i\in I}$
        in this way (where $\{M_i\}_{i\in I}$ is a fixed family of $(R,S)$-bimodules).
        Replace $Q$ with  a quiver consisting of two vertices $\{u,v\}$ and $|I|$ arrows $\{a_i\}_{i\in I}$
        from $u$ to $v$, put $\catC_{a_i}=\lMod{R}$ and define $T_{a_i}(X)=\Hom_S(X,M_i)$. Alternatively,
        any system $\{\omega_i\}_{i\in I}$ is equivalent to an $(R,S)$-bilinear pairing taking values in
        $\prod_{i\in I}M_i$.
    \end{example}

\rem{
    \begin{remark}
        When treating representations of $Q$ on \emph{arbitrary} categorical realizations $\catR$, one may assume that
        $Q$ does not have multiple edges. Indeed, let $\{a_i\}_{i\in I}$ be the set of edges
        from $u$ to $v$ ($u,v\in Q$), and let $Q'$ to be the quiver  obtained from $Q$ by
        replacing $\{a_i\}_{i\in I}$ with a single edge $a$. Define a categorical realization $\catR'=(\{\catC'_x\},\{T'_a,S'_a\})$
        of $Q'$ as follows: For all $x\in Q'_0\cup Q'_1\setminus\{a\}$ define $\catC'_x=\catC_x$
        and set $\catC'_a=\prod_{i\in I}\catC_{a_i}$ (i.e.\
        $\catC'_a$ is the product of the categories $\{\catC_{a_i}\}_{i\in I}$).
        For all $b\in Q'_1\setminus \{a\}$, let $S'_b=S_b$ and $T'_b=T_b$, and set
        $S'_a=\prod_{i\in I}S_{a_i}$ and $T'_a=\prod_{i\in I}T_{a_i}$.
        Then $\Rep(Q,\catR)$ is equivalent to $\Rep(Q',\catR')$.
        Repeating this procedure repeatedly yields a quiver $\quo{Q}$ without
        multiple edges and a categorical realization $\quo{\catR}$ such
        that $\Rep(Q,\catR)\sim\Rep(\quo{Q},\quo{\catR})$.
    \end{remark}
}

    We record for later the following easy proposition.

    \begin{prp}\label{REP:PR:reps-are-pseudo-abelian}
        The category $\Rep(Q,\catR)$ is additive. If moreover $\catC_v$ is pseudo-abelian
        for all $v\in Q_0$, then $\Rep(Q,\catR)$ is pseudo-abelian.
    \end{prp}

    \begin{proof}
        The first assertion is straightforward. To see the second assertion, let $e:=\{e_v\}$ be an idempotent
        endomorphism of $\rho:=(\{C_v\},\{f_a\})\in\Rep(Q,\catR)$. Then each $e_v$ is an idempotent
        endomorphism of $C_v\in\catC_v$ and hence admits a corresponding summand $(X_v, i_v,p_v)$.
        Define $\sigma:=(\{X_v\},\{T_ap_{t(a)}\circ f_a\circ S_ai_{s(a)}\}_{a\in Q_1})$, $i=\{i_v\}$ and $p=\{p_v\}$.
        We claim that $(\sigma,i,p)$ is a summand of $C$ corresponding to $e$. Namely, $i\in\Hom(\sigma,\rho)$, $p\in\Hom(\rho,\sigma)$,
        $i\circ p=e$ and $p\circ i=\id_\sigma=\{\id_{X_v}\}_{v\in Q_0}$. All these assertions follow by a routine
        argument. For example, $i\in\Hom(\sigma,\rho)$ since $T_a i_{t(a)} \circ (T_ap_{t(a)}\circ f_a\circ S_ai_{s(a)})=
        T_a(i_{t(a)}\circ p_{t(a)})\circ f_a\circ S_ai_{s(a)}=T_ae_{t(a)}\circ f_a\circ S_ai_{s(a)}=
        f_a\circ S_ae_{s(a)}\circ S_ai_{s(a)}=f_a\circ S_a(e_a\circ i_{s(a)})=f_a\circ S_ai_{s(a)}$.
    \end{proof}

    \begin{remark}
        If the categories $\{C_x\}_{x\in Q_0\cup Q_1}$ are abelian,
        the functors $\{S_a\}_{a\in Q_1}$ are right exact, and the functors $\{T_a\}_{a\in Q_1}$ are left exact,
        then $\Rep(Q,\catR)$ is abelian. This
        fails if we drop the exactness assumptions, though. We have omitted the straightforward technical verification
        of these facts since
        we do not need them  here.
    \end{remark}

\section{Linearly Topologized Categories}
\label{section:LT}

    Recall that a topological ring $R$ is called \emph{linearly topologized} (abbrev.: LT)
    if $R$ admits a local basis (i.e.\ a basis of neighborhoods of $0$) consisting of two-sided ideals.
    For example, the $p$-adic integers $\Z_p$ are LT, but the $p$-adic numbers $\Q_p$ are not LT. In addition,
    every ring associated with the discrete topology is LT.
    When $R$ is LT, the topology on $R$ is completely determined by the set of open ideals in $R$, denoted
    $\calI_R$. Conversely, every filter base\footnote{
        Let $X$ be a set. A nonempty subset $\calF\subseteq P(X)$ is a filter base
        if $\emptyset\notin \calF$ and for all $A,B\in\calF$, there exists $C\in\calF$
        with $C\subseteq A\cap B$.
    } $\calF$ of ideals of $R$ gives rise to a unique topology on $R$ such
    that $\calF$ is a basis of neighborhoods of $0$; take the topology spanned by  $\{x+J\where x\in R,\, J\in\calF\}$.
    For an extensive discussion about topological rings and proofs of the previous facts,
    see \cite{Wa93}.

    Let $\catC$ be an additive category. Recall that an ideal $I$ of $\catC$
    is a map adjoining to every two objects $X,Y\in\catC$ an additive subgroup
    $I(X,Y)\subseteq \Hom_{\catC}(X,Y)$ such that $\Hom_{\catC}(Y,Y')\circ I(X,Y)\circ \Hom_{\catC}(X',X)\subseteq I(X',Y')$
    for all $X',Y'\in\catC$. For brevity, we write $I(X)$ instead of $I(X,X)$. The latter
    is clearly an ideal of $\End_{\catC}(X)$.

    \smallskip

    A \emph{linearly topologized} category consists of a pair $(\catC,\calI)$
    where $\catC$ is an additive category and $\calI$ is a collection
    of ideals of $\catC$ such that for all $I,J\in \calI$, there exists
    $K\in\calI$ with $K(X,Y)\subseteq I(X,Y)\cap J(X,Y)$ for all $X,Y$.
    We then call $\calI$ a \emph{linear topology} on $\catC$.
    In this case, for all $X,Y\in\catC$, the cosets $\{f+I(X,Y)\where f\in \Hom_{\catC}(X,Y),\,I\in\calI\}$
    span a topology on $\Hom_{\catC}(X,Y)$ which makes it
    into an additive topological group. Furthermore, composition is continuous, hence $\End_{\catC}(X)$
    is an LT ring for all $X\in\catC$. We say
    that $(\catC,\calI)$ is \emph{Hausdorff} if $\bigcap_{I\in\calI}I(X,Y)=0$
    for all $X,Y\in\catC$, i.e.\ the topology
    induced on $\Hom_{\catC}(X,Y)$ is Hausdorff For all $X,Y\in\catC$. If $(\catC',\calI')$ is another LT category,
    we say that an additive functor $F:\catC\to \catC'$ is continuous if for all $X,Y\in \catC$
    and $I'\in\calI'$, there exists $I\in\calI$ such that $F(I(X,Y))\subseteq I'(FX,FY)$. This is equivalent
    to saying that the map $F:\Hom_{\catC}(X,Y)\to \Hom_{\catC'}(FX,FY)$ is continuous for all $X,Y\in\catC$.

    We will often drop $\calI$ from the notation, specifying only $\catC$. In this
    case, we will write $\calI_{\catC}$ to denote the implicit collection $\calI$.

\rem{
    to be \emph{linearly topologized}
    if  for all $C\in\catC$, the ring $\End_{\catC}(C)$ is equipped with a linear ring topology, $\tau_C$,
    and for all $C'\in\End_{\catC}(C)$ the map
    \[
    (f,g)\mapsto \smallSMatII{f}{0}{0}{g}:\End_{\catC}(C)\times\End_{\catC}(C')\to\End_{\catC}(C\oplus C')
    \]
    is a \emph{topological embedding}. In this case,
    the collection $\{\tau_C\}_{C\in\catC}$ is called a \emph{linear topological structure}
    on $\catC$. A linearly topologized category $\catC$ is \emph{Hausdorff}
    if  $\tau_C$ is Hausdorff for all $C\in\catC$. A functor $F:\catC\to \catC'$ between two
    LT categories is \emph{continuous} if the map $F:\End_\catC(C)\to\End_{\catC'}(FC)$ is continuous
    for all $C\in\catC$.
}

    \begin{example}\label{REP:EX:LT-categories-I}
        Any additive category $\catC$ can be made into a Hausdorff LT category
        by taking $\calI=\{0\}$ where $0(X,Y)=\{0_{X,Y}\}$ for all $X,Y\in\catC$. In this case, any
        additive functor from $\catC$ to another LT category is continuous.
    \end{example}

    \begin{example}\label{REP:EX:LT-categories-II}
        Let $R$ be an LT ring. Abusing the notation, for every $I\in\calI_R$ and $M,N\in\rMod{R}$,
        define $I(M,N)=\Hom_R(M,NI)$. Then $(\rMod{R},\calI_R)$ is an LT category.
        We call $\calI_R$ the \emph{standard linear topology} of $\rMod{R}$.

        We say that a module $M\in\rMod{R}$ is \emph{Hausdorff} if $\bigcap_{J\in\calI_R}MJ=0$.
        Let $\rHMod{R}$ denote the full subcategory of $\rMod{R}$
        whose objects are the Hausdorff modules. Then $(\rHMod{R},\calI_R)$
        is a \emph{Hausdorff} LT category.
        Note that $\rHMod{R}$ contains the category of finitely generated projective right $R$-modules, denoted $\rproj{R}$.

        Let $\FP{R}$ denote the category of finitely presented right $R$-modules and let $\HFP{R}$
        denote the category of Hausdorff f.p.\ modules. Then $(\HFP{R},\calI_R)$ is a pseudo-abelian Hasudorff LT
        category containing $\rproj{R}$. The question of when $\FP{R}=\HFP{R}$ was considered in \cite[\S9]{Fi12C}.
        Several cases when this holds are:
        \begin{enumerate}
            \item[(i)] $R$ is an almost fully-bounded noetherian ring whose primitive images are artinian
            (e.g.\ a noetherian PI ring) and $R$ is endowed with the $\Jac(R)$-adic topology (\cite[Th.\ 3.2.28]{Ro88}).
            \item[(ii)] $R$ is \emph{strictly pro-right artinian}, namely, $R$ is the inverse limit
            of discrete right artinian rings $\{R_i\}_{i\in I}$
            and the standard map $R\to R_i$ is surjective for all $i$ (\cite[Cr.\ 9.8]{Fi12C}).
            For example, this holds when $R$ is semilocal, $R=\invlim \{R/\Jac(R)^n\}_{n\in\N}$
            and $\Jac(R)$ is f.g.\ as a right ideal.
            \item[(iii)] $R$ is a complete \emph{rank-$1$} valuation ring\footnote{
                To our purposes, a valuation ring $R$ is of rank $1$ if its \emph{additive} valuation
                takes values in $\R_{\geq 0}$.
            } considered as an LT using its valuation (\cite[Rm.\ 9.7]{Fi12C}).
        \end{enumerate}
    \end{example}

    \begin{example}\label{REP:EX:cont-functor}
        Let $R$ and $S$ be two LT rings and let $Q$ be an $(R,S)$-bimodule
        such that for all $J\in\calI_S$ there is $I\in\calI_R$ such that $IQ\subseteq QJ$.
        Then the functor $-\otimes_RQ:\rMod{R}\to\rMod{S}$ is continuous when both categories are endowed with
        the standard linear topology. Indeed, for $I$ and $J$ as above and $M,N\in\rMod{R}$,
        we have
        $I(M,N)\otimes \id_Q=\Hom_R(M,NI)\otimes\id_Q\subseteq\Hom_S(M\otimes Q,NI\otimes Q)
        \subseteq\Hom_S(M\otimes Q,N\otimes QJ)=J(M\otimes Q,N\otimes Q)$.
    \end{example}

\section{Fitting's Property}
\label{section:pi-regular}

    In this section, we recall the definition of Fitting's Property (abbrev.: FP) and bring several examples.
    Fitting's Property was introduced in \cite[\S5]{Fi12C}
    under the name quasi-$\pi_\infty$-regular. The name was changed in recommendation
    of several experts due to the overusage of the term ``regular''
    in ring theory.

    \begin{dfn}
        A Hausdorff LT ring $R$ satisfies \emph{Fitting's Property} (abbrev.: FP)
        if for all $n\in\N$ and $a\in \nMat{R}{n}$, there is an idempotent
        $e\in \nMat{R}{n}$ such that $a=eae+(1-e)a(1-e)$, $eae$ is invertible in $eRe$ and   $(1-e)a^n(1-e)\to 0$
        as $n\to \infty$.\footnote{
            This resembles Fitting's Lemma, hence the name ``Fitting's Property''.
        } Here, $\nMat{R}{n}$ is endowed with the product topology (view $\nMat{R}{n}$ as $R^{n^2}$).

        If the above condition is satisfied only when $n=1$, then $R$ is said to satisfy the \emph{weak Fitting Property}.
    \end{dfn}

    The idempotent $e$ in the definition turns out to be uniquely determined by the element $a$ (see \cite[Lm.\ 5.6]{Fi12C}).
    We call $e$ the \emph{associated idempotent of $a$}.

    \begin{example}\label{REP:EX:pi-reg-ring}
        (i) Let $R$ be a rank-$1$ valuation ring. Then $R$ is Henselian if and only if it has FP
        as an LT ring;
        see \cite[Prp.\ 8.6]{Fi12C}. Thus, Fitting's Property can be thought of as
        a generalization of Henselianity for topological rings.

        (ii) A non-topological ring $R$ has
        FP when given the discrete
        topology if and only if  it is \emph{(strongly) $\pi_\infty$-regular}\footnote{
            Another name in the literature is ``completely $\pi$-regular''.
        } (\cite[Prp.\ 4.2]{Fi12C}).
        The family of $\pi_\infty$-regular rings include semiprimary and right/left perfect rings\footnote{
            A ring is semiprimary (resp.\ right prefect) if it is semilocal and its Jacobson
            radical is nilpotent (resp.\ right
            T-nilpotent).
        } (\cite[Rm.\ 2.9]{Fi12C}).
        In particular, discrete left/right artinian rings have FP.

        (iii) An LT ring $R$ is called \emph{pro-semiprimary} if it is the inverse limit
        of discrete semiprimary rings.
        For example, $\Z_p$ is pro-semiprimary since $\Z_p\cong \invlim\{\Z/p^n\Z\}_{n\in\N}$ (as topological rings).
        By \cite[Lm.\ 5.14]{Fi12C}, every pro-semiprimary  ring
        has FP.

        (iv) A product of LT rings satisfying FP has FP  when given the
        product topology.
    \end{example}

    Although not needed here, we note that an LT ring with FP
    is semiperfect
    if and only if $R$ does not contain an infinite set of pairwise orthogonal idempotents;
    see \cite[Lm.~5.9]{Fi12C}.

    \medskip

    Let $R$ be a Hausdorff LT ring and let $R_0\subseteq R$ be a subring.
    According to \cite[\S5]{Fi12C}, the subring $R_0$ is called a
    (topologically) \emph{semi-centralizer subring} if there exists a Hausdorff LT ring $S$, containing $R$ as a topological ring, and
    a subset $X\subseteq S$ such that $R_0=\Cent_R(X):=\{r\in R\suchthat rx=xr~ \forall x\in X\}$.
    By \cite[\S5]{Fi12C}, we have:

    \begin{thm}\label{REP:TH:semi-cent-subring}
        Let $R$ be LT ring satisfying FP. Then every semi-centralizer subring of $R$
        has FP (w.r.t.\ the induced topology).
        If moreover $R$ is semiperfect, then every semi-centralizer subring $R_0$ of $R$
        is semiperfect and satisfies $\Jac(R_0)^n\subseteq \Jac(R)$ for some $n$.
    \end{thm}

    The condition $\Jac(R_0)^n\subseteq \Jac(R)$ in the theorem implies that the semi-centralizer subrings
    of semiprimary (resp.\ right perfect) rings are  semiprimary (resp.\ right perfect).

    \smallskip

    We finish with recalling the following facts:

    \begin{prp}\label{REP:PR:inverse-image-of-inv-subring}
        Let $R$ and $S$ be Hausdorff LT rings, let $\vphi:R\to S$ be a continuous ring homomorphism,
        and let $S_0\subseteq S$ be a semi-centralizer subring of $S$. Then $\vphi^{-1}(S_0)$
        is a semi-centralizer subring of $R$.
    \end{prp}

    \begin{proof}
        See \cite[Cr.\ 3.2]{Fi12C}. The proof is phrased for non-topological rings but works
        for topological rings by  \cite[Prp.\ 5.4]{Fi12C}.
    \end{proof}

    \begin{prp}\label{REP:PR:intersection-of-centralizer-subrings}
        The intersection of semi-centralizer subrings is a semi-centralizer subring.
    \end{prp}

    \begin{proof}
        This follows from \cite[Prp.\ 5.4(e)]{Fi12C}.
    \end{proof}

\section{Main Result}
\label{section:main}

    %We now phrase and prove our main result.

    Let $Q$ be a quiver and let $\catR=\Circs{\{\catC_x\}_{x\in Q_0\cup Q_1}, \{S_a,T_a\}_{a\in Q_1}}$
    be a categorial realization of $Q$. We say that $\catR$ is \emph{pseudo-abelian} if
    each of the categories $\{\catC_v\}_{v\in Q_0}$ is psuedo-abelian, and
    say that $\catR$ is \emph{linearly topologized}
    if each of the categories $\{\catC_v\}_{v\in Q_0}$ is linearly topologized.
    In the latter case, we
    make $\Rep(Q,\catR)$ into an LT category as follows: Write $\calI_v=\calI_{\catC_v}$.
    For every $I:=\{I_v\}_{v\in Q_0}$ with $I_v\in\calI_v$
    and $\rho=(\{C_v\},\{f_a\}),\rho'=(\{C'_v\},\{f'_a\})\in\Rep(Q,\catR)$ define
    $I(\rho,\rho')=\Hom(\rho,\rho')\bigcap\prod_{v\in Q_0}I_v(C_v,C'_v)$. Then
    $\{\{I_v\}_{v\in Q_0}\where I_v\in \calI_v\}$ is a linear topology on $\Rep(Q,\catR)$. (In case $Q_0$
    is not assumed to be finite, one should adjust the definition by
    assuming that $I_v=\Hom_{\catC_v}$ for almost all $v\in Q_0$.)
    Notice that the topology on $\Hom(\rho,\rho')$ is the topology induced
    from the the product topology on $\prod_{v\in Q_0}\Hom_{\catC_v}(C_v,C'_v)$.

    Assume $\catR$ is LT. We say that $\catR$ is \emph{Hausdorff} if each of the LT categories
    $\{\catC_v\}_{v\in Q_0}$ is Hausdorff, and for all $a\in Q_1$, $C\in\catC_{s(a)}$ and $C'\in\catC_{t(a)}$
    we have
    \begin{equation}\label{REP:EQ:Hausdorff-cond}\tag{$*$}
        0~=\bigcap_{\substack{I\in\calI_{s(a)} \\ I'\in\calI_{t(a)}}}
        \Big({T_aI'(C')\circ \Hom_{\catC_a}(S_aC,T_aC')+\Hom_{\catC_a}(S_aC,T_aC')\circ S_aI(C)}\Big)\ .
    \end{equation}
    It is easy to see that in this case $\Rep(Q,\catR)$ is also Hausdorff (this is true
    even without the condition \eqref{REP:EQ:Hausdorff-cond}).
    The reason for condition \eqref{REP:EQ:Hausdorff-cond} will be seen in the proof of Theorem~\ref{REP:TH:main}
    below.

    \begin{example}\label{REP:EX:cont-realizations}
        (i) If $\catR$ is the trivial realization
        of $Q$ on a Hausdorff LT category $\catC$, then $\catR$ is Hausdorff (straightforward).

        (ii) Assume that each of the categories $\{\catC_x\}_{x\in Q_0\cup Q_1}$
        is Hausdorff LT and that the functors $\{S_a,T_a\}_{a\in Q_1}$
        are continuous. Then $\catR$ is Hausdorff. Indeed,
        we only need to verify \eqref{REP:EQ:Hausdorff-cond}: Let $a\in Q_1$
        and let $C\in \catC_{s(a)}$, $C'\in\catC_{t(a)}$. Since $S_a$ and $T_a$
        are continuous, for all $J\in\calI_a$, there are $I\in\calI_{s(a)}$ and $I'\in\calI_{t(a)}$
        such that $S_aI(C)\subseteq J(S_aC)$ and $T_aI'(C')\subseteq J(T_aC')$.
        Therefore,
        \begin{align*}
        T_aI'(C')\circ \Hom_{\catC_a}(S_aC,T_aC')+\Hom_{\catC_a}(S_aC,T_aC')\circ S_aI(C) &\subseteq \\
        J(T_aC')\circ\Hom_{\catC_a}(S_aC,T_aC')+\Hom_{\catC_a}(S_aC,T_aC')\circ J(S_aC) &\subseteq \\
        J(S_aC,T_aC')+J(S_aC,T_aC')&= J(S_aC,T_aC')\ .
        \end{align*}
        Since $\catC_a$ is Hausdorff, $\bigcap_{J\in\calI_a}J(S_aC,T_aC')=0$, so \eqref{REP:EQ:Hausdorff-cond} holds.
    \end{example}

    \begin{thm}\label{REP:TH:main}
        Let $Q$ be a quiver and let $\catR$  be a Hausdorff LT pseudo-abelian categorical realization
        such that $\End_{\catC_v}(C)$ is semiperfect and has Fitting's Property for all $v\in Q_0$, $C\in\catC_v$.
        Then $\Rep(Q,\catR)$ is a Hausdorff LT pseudo-abelian category and $\End(\rho)$ is
        semiperfect and has FP for all $\rho\in \Rep(Q,\catR)$.
        In particular, $\Rep(Q,\catR)$ is a Krull-Schmidt category.
    \end{thm}

    \begin{proof}
        That $\Rep(Q,\catR)$ is a Hausdorff LT category is clear. It is pseudo-abelian by Proposition~\ref{REP:PR:reps-are-pseudo-abelian}.
        Thus, it is left to verify that for all $\rho=(\{C_v\},\{f_a\})\in\Rep(Q,\catR)$,
        the endomorphism ring of $\rho$ is semiperfect and has FP.
        Recall that $\End(\rho)$ consists of collections $\{\psi_v\}_{v\in Q_0}$ such that
        $T_a\psi_{t(a)}\circ f_a=f_a\circ S_a\psi_{s(a)}$ for all $a\in Q_1$.

        For every $a\in Q_1$, let $H_a=\Hom_{\catC_a}(S_aC_{s(a)},T_aC_{t(a)})$.
        We make $H_a$ into a $(\End_{\catC_{t(a)}}(C_{t(a)}),\End_{\catC_{s(a)}}(C_{s(a)})$-bimodule
        by setting $f\cdot h\cdot g=T_af\circ h \circ S_ag$
        for all $f\in \End(C_{t(a)})$, $g\in \End(C_{s(a)})$.
        For all
        $I\in\calI_{s(a)}$, $J\in\calI_{t(a)}$,  let
        \[
        K_{I,J}=J(C_{t(a)})\cdot H_a+H_a\cdot I(C_{s(a)})
        \]
        and define
        \begin{align*}
        W_a&=\SMatII{\End(C_{t(a)})}{H_a}{0}{\End(C_{s(a)})}\qquad\text{and} \\
        \calB_a&=\left\{\SMatII{J(C_{t(a)})}{K_{I,J}}{0}{I(C_{s(a)})}\,\Big|\,
        I\in\calI_{s(a)}, J\in\calI_{t(a)}\right\}\ .
        \end{align*}
        Then $W_a$ is a ring and $\calB_a$ is a filter base of ideals of  $W_a$. Thus, $\calB_a$ induces a linear ring
        topology on $W_a$, and that topology is Hausdorff since $\catR$ is Hausdorff ($\bigcap_{I,J}K_{I,J}=0$
        by \eqref{REP:EQ:Hausdorff-cond}).

        Let $R=\prod_{v\in Q_0}\End_{\catC}(C_v)$. Then $R$ is a Hausdorff LT ring since $\catR$ is Hausdorff.
        Define
        $\vphi_a:R\to W_a$ by
        \[
        \vphi_a(\{\psi_v\})=\SMatII{\psi_{t(a)}}{0}{0}{\psi_{s(a)}}\ .
        \]
        and let
        \[
        \alpha_a=\smallSMatII{0}{f_a}{0}{0}\in W_a\ .
        \]
        It is clear that $\vphi_a$ is a continuous homomorphism
        of LT rings.
        Furthermore, by definition, we have $\End(\rho)=\bigcap_{a\in Q_1}\vphi_a^{-1}(\Cent_{W_a}(\alpha_a))$. Since
        $\Cent_{W_a}(\alpha_a)$ is a semi-centralizer subring of $W_a$, $\vphi_a^{-1}(\Cent_{W_a}(\alpha_a))$
        is a semi-centralizer subring of $R$ (Proposition~\ref{REP:PR:inverse-image-of-inv-subring}),
        so
        by Proposition~\ref{REP:PR:intersection-of-centralizer-subrings}, $\End(\rho)$ is a semi-centralizer subring of $R$.
        The ring $R$
        is semiperfect and has FP by assumption. Thus, by Theorem~\ref{REP:TH:semi-cent-subring}, $\End(\rho)$ is semiperfect
        and has FP.
    \end{proof}

    \begin{remark}\label{REP:RM:remarks-on-main-thm}
        (i) If we drop the assumption that $Q_0$ is finite and that the endomorphism rings of objects in $\catC_v$ ($v\in Q_0$)
        are semiperfect in Theorem~\ref{REP:TH:main}, then the endomorphism rings of objects  in $\Rep(Q,\catR)$  still have FP.

        (ii) When the assumptions of Theorem~\ref{REP:TH:main}
        are satisfied, the indecomposable representations in  $\Rep(Q,\catR)$ are precisely those
        with local endomorphism ring. (Indeed, an object in an additive category is indecomposable precisely
        when its endomorphism ring has no nontrivial idempotents, and a semiperfect ring with nontrivial
        idempotents is local.)
    \end{remark}

    \begin{example}
        Theorem~\ref{REP:TH:main} fails if we do not assume  condition~\eqref{REP:EQ:Hausdorff-cond}:
        Let $R$ and $S$ be two semiperfect LT rings satisfying
        FP, and assume that there is a \emph{non-topological} ring $M$ containing $R$
        and $S$ such that $R\cap S$ is not semiperfect. This implies that $R\cap S$ does not have FP
        with respect to any topology, for otherwise, $R\cap S$ would be semiperfect by \cite[Lm.~5.9(ii)]{Fi12C}.

        View $M$ as an $(R,S)$ bimodule and define $Q$ and $\catR$ as in Example~\ref{REP:EX:bilinear-pairings}
        with the difference that only projective $R$- and $S$-modules are allowed. Namely,
        $Q$ is the quiver $u\xrightarrow{a}v$, $\catC_u=\lproj{R}$, $\catC_v=(\rproj{S})^\op$, $\catC_a=\lMod{R}$, $S_a=\id$ and
        $T_a=\Hom_S(-,M)$. Then $\catC_u$ and $\catC_v$ are Hausdorff with respect to their standard linear
        topology.
        Let $\rho=(\{C_u,C_v\},\{f_a\})\in\Rep(Q,\catR)$ be the object corresponding to the $(R,S)$-bilinear
        pairing $\omega:R\times S\to M$ given by $\omega(r,s)=rs$. That is, $C_u={}_RR$, $C_v=S_S$ and $f_a:{}_RR\to\Hom_S(S_S,M_S)\cong M$
        is just the inclusion map. In this case, $H_a=\Hom_R({}_RR,{}_RM)\cong M$ as an $(R,S)$-bimodule
        and under that isomorphism, $f_a$ is mapped to $1$. Following the proof of Theorem~\ref{REP:TH:main},
        we now see that $\End(\rho)$ is isomorphic to
        \[
        \left\{\smallSMatII{r}{0}{0}{s}\in \smallSMatII{R}{M}{0}{S}\where  \smallSMatII{r}{0}{0}{s}\smallSMatII{0}{1}{0}{0}=
        \smallSMatII{0}{1}{0}{0}\smallSMatII{r}{0}{0}{s}\right\}\cong R\cap S
        \]
        which is not semiperfect nor has FP.
        An explicit example of $R$, $S$ and $M$ as above is $R=\Z_p$, $S=\Z_q$ and $M=\Q_p\otimes_{\Q} \Q_q$.
        Here, $p,q$ are distinct prime numbers and $\Z_p$ and $\Z_q$ are identified with $\Z_p\otimes 1$ and $1\otimes \Z_q$.
        Since $a\otimes 1=1\otimes b$ implies $a=b\in\Q$, we see that $R\cap S=\{nm^{-1}\where n\in \Z,\,m\in\Z\setminus(p\Z\cup q\Z)\}$,
        a  non-semiperfect ring.
    \end{example}

    The rest of this section concerns applications
    of Theorem~\ref{REP:TH:main}. Recall that for an LT ring $R$,
    $\FP{R}$ denotes the category of finitely presented right $R$-modules
    and $\HFP{R}$ denotes the category of Hausdorff finitely presented right $R$-modules.

    \begin{example}\label{REP:EX:when-main-applies}
        Let $\catR$ be the trivial realization of $Q$ over an additive category $\catC$ (i.e.\ $\catC_x=\catC$
        and $S_a=T_a=\id_{\catC}$ for all $x\in Q_0\cup Q_1$, $a\in Q_1$).
        The assumptions of Theorem~\ref{REP:TH:main} are satisfied for $\catR$ when $\catC$ is a Hausdorff LT category
        whose endomorphisms are semiperfect and has FP. These conditions are satisfied by the following categories:

        (i) The category $\HFP{R}$, where $R$ is a semiperfect
        Hausdorff LT ring satisfying Fitting's Property. Indeed, by \cite[Th.\ 8.3(ii)]{Fi12C},
        the endomorphism ring of every object in $\HFP{R}$ is semiperfect and has FP.
        Note that $\HFP{R}$ contains all f.g.\ projective right  $R$-modules.

        (ii) The category $\FP{R}$ where $R$ is a strictly pro-right-artinian ring
        or a rank-$1$ Henselian valuation ring. This follows from (i) and
        Example~\ref{REP:EX:LT-categories-II}.

        (iii) The category $\Rep(Q',\catR')$ where $\catR'$ is a categorical realization of $Q'$
        satisfying the assumptions of
        Theorem~\ref{REP:TH:main}.
    \end{example}

    \begin{example}
        Let $R$ be a commutative LT ring with FP  and let $M=\{M_a\}_{a\in Q_1}$ be a family of
        modules in $\HFP{R}$. Then the Krull-Schmidt Theorem holds for the category of $M$-twisted representations
        of $Q$ on $\HFP{R}$ in the sense of \cite[\S2]{GothKing05}. Indeed, apply Theorem~\ref{REP:TH:main}
        with the categorical realization of Example~\ref{REP:EX:twisted-quivers}. The assumptions
        of Theorem~\ref{REP:TH:main} apply since the functors
        $\{T_a,S_a\}$ are continuous (Example~\ref{REP:EX:cont-functor})
        and the categories $\{\catC_x\}_{x\in Q_0\cup Q_1}$ are Hasudorff LT (Example~\ref{REP:EX:when-main-applies}(i)).
    \end{example}

    \begin{example}
        Let $R$ be a semiperfect LT ring satisfying FP. Then every endomorphism $f$ in $\HFP{R}$
        decomposes to a direct sum of indecomposable endomorphism $f=f_1\oplus \dots\oplus f_t$
        and this decomposition is essentially unique, i.e.\ every other such decomposition is conjugate to
        the original one. (This decomposition can be regarded as a ``Jordan decomposition'' of $f$.)
        To see why this holds, recall from Example~\ref{REP:EX:examples-of-quivers}
        that the category of endomorphisms in $\HFP{R}$
        can be understood as $\Rep(Q,\HFP{R})$ for suitable $Q$. Now apply Theorem~\ref{REP:TH:main}
        together with the Krull-Schmidt Theorem; the assumptions of Theorem~\ref{REP:TH:main} hold
        by Example~\ref{REP:EX:when-main-applies}(i).
    \end{example}

    \begin{example}\label{REP:EX:semiprimary-endo-ring}
        Let $\catR$ be a non-linearly-topologized semi-abelian categorical realization of $Q$.
        Assume that $\End_{\catC_v}(C)$ is semiprimary for all $v\in Q_0$ and $C\in\catC_v$.
        Then the endomorphism rings of all objects in $\Rep(Q,\catR)$ are semiprimary.
        Indeed, give the categories $\{\catC_x\}_{x\in Q_0\cup Q_1}$ the discrete topology
        of Example~\ref{REP:EX:LT-categories-I}. By Example~\ref{REP:EX:cont-realizations},
        $\catR$ is Hausdorff, and by Example~\ref{REP:EX:pi-reg-ring}(ii),
        the endomorphism rings of objects in $\catC_v$ are semiprefect and have FP.
        Thus, we can apply Theorem~\ref{REP:TH:main} to $\catR$.
        In the proof of the theorem, it was shown that the endomorphism ring
        of $\rho=(\{C_v\},\{f_a\})\in\Rep(Q,\catR)$ is a semi-centralizer subring
        of $\prod_{v}\End_{\catC_v}(C_v)$, which is semiprimary in our case.
        Therefore, by the comment after Theorem~\ref{REP:TH:semi-cent-subring}, $\End(\rho)$
        is semiprimary. (The same argument works if we replace semiprimary with right perfect.)

        Examples of categories $\catC_v$ with semiprimary endomorphisms rings include
        abelian categories consisting of objects of finite length (Example~\ref{REP:EX:abelian-cats})
        and the category $\FP{R}$
        when $R$ is a semiprimary ring (this is due to Bjork \cite[Th.\ 4.1]{Bj71B};
        see \cite[Ths.\ 7.3 \& 8.3]{Fi12C} for generalizations).
        When the the categories $\{\catC_x\}_{x\in Q_0\cup Q_1}$
        are  abelian, consist of objects of finite length, and the functors $\{S_a\}$ and $\{T_a\}$ are right-exact
        and left-exact, respectively, the category $\Rep(Q,\catR)$ is abelian and consists
        of objects of finite length. Therefore, can use Example~\ref{REP:EX:abelian-cats} rather
        than Theorem~\ref{REP:TH:main} to prove that endomorphism rings in $\Rep(Q,\catR)$
        are semiprimary.
\rem{

        In the special case that $\catR$ is the trivial realization of $Q$ on an \emph{abelian}
        category $\catC$ in which all objects have finite length, there
        is no need to use Theorem~\ref{REP:TH:main} to get
        the previous result. The reason is that in this case, the category $\Rep(Q,\catR)$
        is abelian and its object have finite length, and the Hadara-Sai Lemma
        (e.g.\ see \cite[Prp.\ 2.9.29]{Ro88}\footnote{
            The proof is phrased for modules
            but works in arbitrary abelian categories
        }) implies that in  such a category, the endomorphism ring of every indecomposable
        object is semiprimary. However, in general, the category $\Rep(Q,\catR)$
        is not abelian even when all the categories $\{\catC_x\}_{x\in Q_0\cup Q_1}$
        are abelian and all the functors $\{S_a,T_a\}_{a\in Q_1}$ are exact.
}
    \end{example}

    \begin{cor}
        Let $R,S$ be semiperfect LT rings satisfying FP and let $M$ be an $(R,S)$-bimodule
        such that $\bigcap_{I,J}(IM+MJ)=0$
        where $I$ and $J$ range over $\calI_R$ and $\calI_S$, respectively.
        Then every $(R,S)$-bilinear pairing $\omega:A\times B\to R$
        with $A\in\lHFP{R}$ and $B\in\HFP{S}$ can be written as an orthogonal sum of
        indecomposable bilinear parings $\omega=\omega_1\perp\dots\perp \omega_t$
        and this decomposition is essentially unique. Namely, if $\omega=\omega'_1\perp\dots\perp \omega'_{t}$
        with $\omega'_1,\dots,\omega'_t$ indecomposable,
        then $t=t'$, and after suitable reordering, $\omega_i$ is isometric to $\omega_i$ for all $1\leq i\leq t$.\footnote{
            Two $(R,S)$-bilinear pairings $\omega_i:A_i\times B_i\to M$ ($i=1,2$) are isometric
            if there are isomorphisms $f:A_1\to A_2$ and $g:B_1\to B_2$ such that $\omega_2(fa,gb)=\omega_1(a,b)$
            for all $a\in A_1$, $b\in B_1$.
        }
    \end{cor}

    \begin{proof}
        Define $\catR$ and $Q$ as in Example~\ref{REP:EX:bilinear-pairings},
        but take $\catC_u$ and $\catC_v$ to be $\lHFP{R}$ and $\HFP{S}$ instead of
        $\lMod{R}$ and $\rMod{S}$.
        Then the objects of $\Rep(Q,\catR)$ correspond to $(R,S)$-bilinear pairings
        $\omega:A\times B\to M$ with $A\in\HFP{R}$ and $B\in\HFP{S}$.

        Recall that $Q$ is the quiver $u\xrightarrow{a} v$.
        The category $\catC_u=\HFP{R}$ is Hausdorff LT (when given the standard topology)
        by definition. Likewise, the category $\catC_v=(\HFP{S})^\op$ inherits the standard
        (Hausdorff) linear structure of $\HFP{S}$.
        We claim that $\catR$ is LT Hausdorff, namely, that \eqref{REP:EQ:Hausdorff-cond} holds.
        Let $I\in\calI_R$ and $J\in\calI_S$.
        Then for all $A\in\lMod{R}$, $B\in\rMod{S}$, we have
        \begin{align*}
        \Hom_{\catC_a}(S_aA,T_aB)\circ S_aI(A) &=  \Hom_R(A,\Hom_S(B,M))\circ I(A)  \\
        & \subseteq  I(A,\Hom_S(B,M))=\Hom_R(A,I\cdot\Hom_S(B,M)) \\
        & \subseteq \Hom_R(A,\Hom_S(B,IM))
        \end{align*}
        and
        \begin{align*}
        T_aJ(B)\circ \Hom_{\catC_a}(S_aA,T_aB) & = \Hom_S(-,M)(J(B,B))\circ\Hom_R(A,\Hom_S(B,M)) \\
        & \subseteq \Hom_R(A,\Hom_S(B,MJ)) \ .
        \end{align*}
        The last inclusion holds since for all $f\in J(B,B)$, $g\in\Hom_R(A,\Hom_S(B,M))$ and $a\in A$,
        we have $(\Hom_S(f,M)\circ g)a=(ga)\circ f\in J(B,M)=\Hom_S(B,MJ)$.
        By assumption, $\bigcap_{I,J}(IM+MJ)=0$, so the previous two equations imply \eqref{REP:EQ:Hausdorff-cond}.

        Next, observe that By Example~\ref{REP:EX:when-main-applies}(i), each of objects in $\catC_u$ and $\catC_v$
        has a semiperfect endomorphism ring satisfying FP. Thus, by Theorem~\ref{REP:TH:main},
        the objects of $\Rep(Q,\catR)$ have a Krull-Schmidt decomposition.
        We are therefore finished if we show that direct sum decompositions in $\Rep(Q,\catR)$
        correspond to orthogonal sum decompositions of $(R,S)$-bilinear pairings. (This is not obvious since
        not all morphisms in $\Rep(Q,\catR)$ correspond to morphisms of $(R,S)$-bilinear pairings; see Example~\ref{REP:EX:bilinear-pairings}.)
        This is routine and is left to the reader.
        \rem{
        Indeed, if $\omega=\omega_1\oplus \omega_2$ with $\omega_i:A_i\times B_i\to M$,
        then }
    \end{proof}

\section{Quivers with Relations and Categories}
\label{section:Fitting}

    In this section we consider quivers with relations and thus restrict our
    discussion to \emph{trivial} realizations of quivers on a fixed additive category $\catC$.

\medskip

    For a quiver $Q$, denote by $P(Q)$ the set of paths in $Q$.
    For $(\{C_v\},\{f_a\})\in \Rep(Q,\catC)$ and a path $p=(v_1,a_1,v_2,a_2,\dots,v_n)$, we set $f_p=f_{a_{n-1}}\circ\dots\circ f_{a_{1}}$
    if $n>1$
    and $f_{p}=\id_{C_{v_1}}$ if $n=1$.
    A \emph{path related quiver} is a quiver $Q$
    equipped with a family of pairs $\frakR=\{p_i,q_i\}_{i\in I}$ such that $p_i$ and $q_i$ are
    paths in $Q$ admitting a common source and target.
    A representation of $(Q,\frakR)$ on $\catC$, is a representation $(\{C_v\},\{f_a\})\in \Rep(Q,\catC)$
    such that $f_{p_i}=f_{q_i}$ for all $i\in I$.
    The class of representations of $(Q,\frakR)$ on $\catC$ forms a full subcategory of $\Rep(Q,\catC)$
    and thus, Theorem~\ref{REP:TH:main} remains true if we replace $\Rep(Q,\catC)$ with that category.

    Every category $\catU$ with finitely many
    objects can be considered as a path related quiver. Indeed, take $Q_0$ to be the objects of $\catU$,
    let $Q_1$ be the morphisms of $\catU$, and take $\frakR$ to be  the relations induced from the composition rule.
    In this case, a representation of $(Q,\frakR)$ on $\catC$ is merely a functor from $\catU$ to $\catC$.
    Theorem~\ref{REP:TH:main} now reads as:

    \begin{thm}\label{REP:TH:main-for-cat}
        Let $\catU$ be a category with finitely many objects and let $\catC$  be a Hausdorff LT pseudo-abelian category
        such that $\End_{\catC}(C)$ is  semiperfect and has FP for all $C\in\catC$.
        Then the category of functors  $F:\catU\to \catC$,
        denoted $\Hom(\catU,\catC)$, is a Hausdorff LT pseudo-abelian category and $\End(F)$ is
        semiperfect and has FP for all $F\in \Hom(\catU,\catC)$.\footnote{
            Recall that an endomorphism of
            a functor $F$ is just a \emph{natural transformation} from $F$ to itself. Namely,
            it is a collection $\{t_U\}_{U\in\catU}$ with $t_U\in\End_{\catC}(FU)$
            such that for every
            $U,V\in \catU$ and $f\in\Hom_{\catU}(U,V)$, one has $t_V\circ Ff=Ff\circ t_U$.
            When $\catU$ is a small category (i.e.\ the objects of $\catU$ form a set), $\End(F)$ is a set.
        }
        In particular, $\Hom(\catU,\catC)$ is a Krull-Schmidt category.
    \end{thm}

    In case $\catU$ is an infinite small category, the same argument  shows that
    $\End(F)$ has FP for every functor $F:\catU\to \catC$.
    While this statement is meaningless in case $\catU$ is a large category
    (since $\End(F)$ is not necessarily a set),
    we can still get a statement resembling Fitting's Lemma:

    \begin{thm}\label{REP:TH:Fitting}
        Let $\catU$ be any category and let $\catC$ be an additive Hausdorff LT pseudo-abelian category
        such that $\End_{\catC}(C)$  has FP for all $C\in \catC$.
        Let $F:\catU\to \catC$ be a functor and let $t:F\to F$ be a natural transformation.
        Then there are functors $F_0,F_1:\catU\to \catC$ and natural transformations
        $t_0:F_0\to F_0$, $t_1:F_1\to F_1$ such that $F=F_0\oplus F_1$ (i.e.\ $FU=F_0U\oplus F_1U$ for all
        $U\in \catU$), $t=t_0\oplus t_1$ (i.e.\ $t_U=(t_0)_U\oplus (t_1)_U$ for all
        $U\in \catU$),
        $t_1$ is a natural isomorphism, and $((t_0)_U)^n\to 0$ as $n\to\infty$ in $\End_\catC(F_0U)$
        for all $U\in\catU$.
        The decomposition $F=F_0\oplus F_1$ is unique up to natural isomorphism.\footnote{
            If $\catC$ is not small, the proof requires a stronger version of the axiom of choice, namely, that there exists a choice
            function for every \emph{class} of nonempty classes. However, this can be avoided if there is a \emph{canonical}
            way to choose the summand corresponding to an idempotent morphism in $\catC$, which is indeed the case
            for many standard categories (e.g.\ module categories).
        }
    \end{thm}

    \begin{proof}
        Let $\catU'$ be a finite full subcategory of $\catU$, and let $F'=F|_{\catU'}$ and $t'=t|_{\catU'}$.
        By Theorem~\ref{REP:TH:main-for-cat} and Remark~\ref{REP:RM:remarks-on-main-thm}(i),
        $\End(F')$ is a has FP (when endowed with the topology induced from
        the product topology on $\prod_{U\in\catU}\End_{\catC}(F'U)$).
        Therefore, $t'$ admits an associated idempotent $e'\in\End(F')$ (see section  \ref{section:pi-regular}),
        namely,
        $t'=e't'e'+(1-e')t'(1-e')$, $e't'e'$ is invertible in $e'\End(F')e'$ and
        $(1-e')t'^n(1-e')\to 0$ when $n\to\infty$.

        Let $\catU''$ be another finite full subcategory of $\catU$. Define $F''$ and $e''$ analogously to
        $F'$ and $e'$.
        If $\catU''$ contains $\catU'$, then $e''|_{\catU'}$ is also an associated idempotent for $t'$,
        so the the uniqueness of the associated idempotent implies $e''|_{\catU'}=e'$.
        Using this, we may define a natural transformation $e:F\to F$ by letting $e_U$ to be $e'_U$
        where $e'$ is obtained from some finite subcategory $\catU'$ as above that contains $U$.
        It is clear that $e$ is idempotent, $t=ete+(1-e)t(1-e)$, and for all $U\in\catU$, $e_Ut_Ue_U$
        is invertible in $e_U\End_{\catC}(FU)e_U$, and $(\id_{FU}-e_U)t_U^n(\id_{FU}-e_U)\to 0$ when $n\to\infty$
        in $\End_{\catC}(FU)$.
        Furthermore, $e$ is  unique.

        Now, define $F_0,F_1:\catU\to \catC$ by taking $F_1U$ and $F_0U$ to be summands of $FU$
        corresponding to the idempotents $e_U$ and $\id_{FU}-e_U$, respectively (here we assume the  axiom of choice
        for \emph{classes}). (More precisely, for all $U\in\catU$, \emph{choose}  summands
        $(S_U,i_U,p_U)$, $(S'_U,i'_U,p'_U)$ corresponding
        to $e_U$, $\id_{FU}-e_U$, respectively. For all $U,V\in\catU$ and
        $f\in\Hom_{\catU}(U,V)$, let $F_1U=S_U$, $F_0U=S'_U$ and
        $F_1f=p_V\circ Ff\circ i_U$, $F_0f=p'_V\circ Ff\circ i'_U$. Now,
        identify $F_1U\oplus F_0U=S_U\oplus S'_U$ with $FU$ via $i_U\oplus i'_U$.) Since $t_U$ commutes with $e_U$,
        it is a direct sum of two morphisms
        $(t_0)_U\in\End_{\catC}(F_0U)$ and $(t_1)_U\in\End_{\catC}(F_1U)$.
        It is easy to see that $F_0,F_1,t_0,t_1$ satisfy all requirements.
        That the decomposition $F=F_0\oplus F_1$ is unique up to natural isomorphism follows
        from the uniqueness of $e$.
    \end{proof}

    \begin{example}
        Let $\catC$ be the category of finite abelian groups, which we make into a Hausdorff LT
        category by endowing it with the discrete topology ($\calI_{\catC}=\{0\}$).
        Since the endomorphism ring of any object in $\catC$ is finite, it has FP.
        Therefore, by Theorem~\ref{REP:TH:Fitting}, for every functor $F:\catC\to \catC$ and a natural
        transformation $t:F\to F$, we have an essentially unique decomposition $F=F_0\oplus F_1$, $t=t_0\oplus t_1$
        where $t_1:F_1\to F_1$ is a natural isomorphism and $t_0:F_0\to F_0$ is a natural transformation which
        is nilpotent on every object.
        For example, in case $F$ is the identity functor and $t$ is multiplication by $2$, the functors $F_0$ and
        $F_1$ are given by $F_0A=\{a\in A\where \text{$\ord(a)$ is a power of $2$}\}$
        and $F_1A=\{a\in A\where \textrm{$\ord(a)$ is odd}\}$ (where $\ord(a)$ denotes the order of $a$ in $A$).
    \end{example}

\section{A Remark on Cancelation}
\label{section:cancelation}

    Let $\catC$ be an additive  category. An object $C\in\catC$
    is said to \emph{cancel from direct sums} if $C\oplus C'\cong C\oplus C''$ implies $C'\cong C''$
    for all $C',C''\in\catC$.
    It is clear that all objects in a Krull-Schmidt category cancel from direct sums.
    In particular, this holds for the objects of $\Rep(Q,\catR)$ (with $Q$ a quiver and $\catR$
    a realization)
    if the assumptions of Theorem~\ref{REP:TH:main} apply.
    However, it turns out that these weaker condition holds under the much milder assumption
    that all objects in the categories $\{\catC_v\}_{v\in Q_0}$ are semilocal ($\catR$ does not have
    to be linearly topologized).

    The latter follows by applying two known results: The first is a
    theorem by Camps and Dicks (\cite{CaDi93}) stating that a subring $S$ of a semilocal ring
    $R$ satisfying $S\cap\units{R}=\units{S}$ (i.e.\ elements of $S$ that are invertible in $R$
    are invertible in $S$) is semilocal. The second is the fact that modules
    with semilocal endomorphism ring cancel from direct sum.
    It is due to Evans and Bass: Evans
    (\cite[Th.\ 2]{Evans73}) showed that if $M$ is an $R$-module and $1$ is \emph{in the stable range} of $\End_R(M)$,
    then $M$ cancels from direct sums; Bass (\cite[Lm.\ 6.4]{Ba64}) showed that $1$ is in the stable
    range of any semilocal ring. The Bass-Evans Theorem actually implies that in an additive category
    $\catC$, every object $C$ whose endomorphism ring is semilocal cancels from direct sums.
    Indeed, assume $C\oplus C'\cong C\oplus C''$ and let $E=C\oplus C'\oplus C''$
    and $R=\End_{\catC}(E)$. It is well-know that the functor $\Hom_{\catC}(E,-):\catC\to \rMod{R}$
    is faithful and full once restricted to $\catC|_E$, the full subcategory of $\catC$ consisting of summands
    of $E^n$ (with $n$ arbitrary). Since $C,C',C''\in\catC|_E$, we may replace them with
    $\Hom(E,C),\Hom(E,C'),\Hom(E,C'')$ and assume $\catC=\rMod{R}$. By Bass and Evans, we have $C'\cong C''$
    in this case.

    Now let $\rho=(\{\catC_v\},\{f_a\})\in\Rep(Q,\catR)$
    and assume $\End_{\catC_v}(C_v)$ is semilocal for all $v\in Q_0$.
    Then $R:=\prod_{v}\End_{\catC_v}(C_v)$ is semilocal. It is straightforward to check that every
    element of $\End(\rho)$ that is invertible in $R$ is invertible in $\End(\rho)$.
    Thus, by Camps and Dicks, $\End(\rho)$ is semilocal. By Bass and Evans, this means $\rho$ cancels from direct
    sums, as required.

    Note that, as in section~\ref{section:main}, the result just established can be applied to obtain
    cancelation of representations of quivers on additive categories, twisted representations
    (in the sense of \cite[\S2]{GothKing05}) and bilinear pairings. One
    should use the fact
    that the endomorphism ring of a finitely generated (resp.\ finitely presented)
    module over a commutative (resp.\ arbitrary) semilocal ring is semilocal
    by \cite[Pr.\ 3.1]{FacHer06} (resp.\ \cite[Pr.\ 3.2]{FacHer06}).
    The details are left to the reader.

\bibliographystyle{plain}
\bibliography{MyBib}

\end{document}